\documentclass[11pt]{amsart}
\usepackage{oldgerm}
\usepackage{amssymb} 
\usepackage{mathrsfs} 
\usepackage{amsmath}
\usepackage{latexsym}
\usepackage[all]{xy}
\usepackage[dvips]{graphics}
\usepackage{amsthm}
\usepackage{enumerate}
\usepackage{bm}

\theoremstyle{plain} 
\newtheorem{thm}{Theorem}[section] 
\newtheorem{prop}[thm]{Proposition} 
\newtheorem{lem}[thm]{Lemma} 
\newtheorem{cor}[thm]{Corollary} 
\theoremstyle{definition} 
\newtheorem{rem}[thm]{Remark} 
\theoremstyle{remark} 
\title[Rational curves on a smooth Hermitian surface]
{
Rational curves on a smooth\\ Hermitian surface
}

\author
{Norifumi Ojiro
}
\address{Department of Mathematics, 
Graduate School of Science, 
Hiroshima University \\
1-3-1, Kagamiyama, Higashi-Hiroshima, Hiroshima, 739-8526, Japan}
\email{d153696@hiroshima-u.ac.jp}

\date{}


\begin{document}
\maketitle

\begin{abstract}
We study the set $R$ of nonplanar rational curves of degree $d<q+2$ on a smooth Hermitian surface $X$ of degree $q+1$ defined over an algebraically closed field of characteristic $p>0$, where $q$ is a power of $p$. We prove that $R$ is the empty set when $d<q+1$. In the case where $d=q+1$, we count the number of elements of $R$ by showing that the group of projective automorphisms of $X$ acts transitively on $R$ and by determining the stabilizer subgroup. In the special case where $X$ is the Fermat surface, we present an element of $R$ explicitly.
\end{abstract}

\section{Introduction}
Let $q$ be a power of a prime $p$, and $k$ an algebraic closure of the finite field $\mathbb{F}_{q}$. For a matrix $m$ with entries in $k$, we denote by $m^{(q)}$ the matrix whose entries are the $q$-th power of those of $m$. We denote by a column vector $\bm{x}={}^\mathrm t\hspace{-0.5mm}(x_0,x_1,x_2,x_3)$ a point in the $k$-projective space $\mathbb{P}^3$. Let $A$ be a nonzero $4$-by-$4$ matrix with entries in $k$. A $k$-Hermitian surface $X_A$ is
defined by
\begin{equation*}
X_A:=\{\bm{x}\in\mathbb{P}^3\ |\ {}^\mathrm t\hspace{-0.5mm}\bm{x}A\bm{x}^{(q)}=0\}.
\end{equation*}
If $A$ is a Hermitian matrix, namely $A$ has the entries in $\mathbb{F}_{q^2}$ and ${}^\mathrm t\hspace{-0.5mm} A=A^{(q)}$, the surface $X_A$ is called a Hermitian surface. 
It is easily shown that $X_A$ is smooth if and only if $A$ is invertible.

The geometry of Hermitian varieties was systematically investigated by B. Segre in \cite{Se}.
Especially, the number of linear spaces lying on a Hermitian variety and their configuration were considered.
It was shown that the numbers of points and lines on a smooth Hermitian surface in $\mathbb{P}^3(\mathbb{F}_{q^2})$ are equal to $(q^3+1)(q^2+1)$ and $(q^3+1)(q+1)$ respectively, and no plane is contained. Further, the set of points and lines on a smooth Hermitian surface forms a block design, see also \cite{BC}. In recent years, the number of rational normal curves totally tangent to a smooth Hermitian variety $X$ has been determined in \cite{Sh} by considering the action of the automorphism group of $X$ on the set of the curves. In \cite{Sh2}, non-singular conics totally tangent to the smooth Hermitian curve of degree $6$ in characteristic $5$ were utilized for a geometric construction of strongly regular graphs. On the other hand, projective isomorphism classes of degenerate Hermitian varieties of corank $1$ and the automorphism group of each isomorphism class have been determined in \cite{Ho}.

Let $A$ be an invertible $4$-by-$4$ matrix with entries in $k$.
We will be concerned with rational curves of degree $>1$ on a smooth $k$-Hermitian surface $X_A$.
Let $d$ be a positive integer and $F$ a $4$-by-$(d+1)$ matrix of ${\rm rank}(F)\ge2$ with entries in $k$. A rational curve $C_F$ of degree $d$ in $\mathbb{P}^3$ is the image of a rational map
\begin{equation}
\mathbb{P}^1\ni{}^\mathrm t\hspace{-0.5mm}(s,t)\longmapsto F\ {}^\mathrm t\hspace{-0.5mm}(s^d,s^{d-1}t,\dots,st^{d-1},t^d)\in\mathbb{P}^3.\label{lem1}
\end{equation}
We call ${\rm rank}(F)$ the rank of the curve $C_F$.
If ${\rm rank}(F)=2$,
then $C_{F}$ degenerates
to a line. If ${\rm rank}(F)=3$,
then $C_{F}$ degenerates to a plane curve of degree $\ge2$.
When
${\rm rank}(F)=4$,
the curve $C_F$ is nondegenerate and is a space curve of degree $\ge3$.
Then $C_F$ is said to be nonplanar, namely $C_F$ is not contained in any plane.
Thus the study of rational curves of rank $2$ on $X_A$ is reduced to that of lines on $X_A$.
Further, an algebraic curve of rank $3$ on $X_A$ is
a smooth $k$-Hermitian curve of degree $q+1$, which is of genus $q(q-1)/2>0$.
Hence we may restrict ourselves to the case of rank $4$.

Our results are as follows:
\begin{thm}\label{thm0}
There is no nonplanar rational curve of
degree $\le q$ on a smooth $k$-Hermitian surface.
\end{thm}

Let $R$ be the set of nonplanar rational curves of degree $q+1$ on a smooth $k$-Hermitian surface $X_A$. 
As will be seen later, the set $R$ is nonempty and each element is projectively isomorphic over $k$ to the smooth curve
\begin{equation*}
C_0:=\left\{{}^\mathrm t\hspace{-0.5mm}(s^{q+1},s^qt,st^q,t^{q+1})\in\mathbb{P}^3\ |\ {}^\mathrm t\hspace{-0.5mm}(s,t)\in\mathbb{P}^1\right\}.
\end{equation*}

We denote by ${\rm Aut}(X_A)$ the group of projective automorphisms of $X_A$. 
Let $n$ be a positive integer.
We deal with the group ${\rm PGU}_n(\mathbb{F}_{q^2})$ defined by
\begin{equation*}
\{Q\in{\rm GL}_n(\mathbb{F}_{q^2})\ |\ {}^\mathrm t\hspace{-0.5mm} QQ^{(q)}=I\}/\bm{\mu}_{q+1}I,
\end{equation*}
where $\bm{\mu}_{q+1}$ denotes the group of $(q+1)$-th roots of unity and $I$ denotes the unit matrix.
As is well-known, the group ${\rm Aut}(X_A)$ is isomorphic to ${\rm PGU}_4(\mathbb{F}_{q^2})$. Then we shall prove the following theorem.
\begin{thm}\label{thm1}
The group ${\rm Aut}(X_A)$ acts transitively on the set $R$, and the stabilizer subgroup is isomorphic
to
${\rm PGU}_2(\mathbb{F}_{q^4})$.
\end{thm}
By Theorem \ref{thm1},
the cardinality of $R$ is equal to $|{\rm PGU}_4(\mathbb{F}_{q^2})|/|{\rm PGU}_2(\mathbb{F}_{q^4})|$.
We know by \cite[pp.64-65]{HT} that
\begin{equation*}
|{\rm PGU}_4(\mathbb{F}_{q^2})|=q^6(q^4-1)(q^3+1)(q^2-1)\ {\rm and}\ |{\rm PGU}_2(\mathbb{F}_{q^4})|=q^2(q^4-1).
\end{equation*}
Thus we have the following.
\begin{cor}\label{cor1}
$|R|=q^4(q^3+1)(q^2-1)$.
\end{cor} 
The number $|R|$ is $432$, $18144$, $249600$, $1890000$, $39645312$, $383162400,\dots$ as $q=2$, $3$, $4$, $5$, $7$, $9,\dots$.

In the special case where $A=I$, that is, where the surface $X_A$ is the Fermat surface, we can explicitly give an element $C_{F_J}$ of $R$ such as
\begin{equation*}
\left.\left\{{}^\mathrm t\hspace{-1.0mm}\left(\eta^{-q}\xi^qs^{q+1}-\eta^{-q}t^{q+1},\ s^qt,\ st^q,\ \omega\eta^{-1}\xi s^{q+1}+\omega\eta^{-1}t^{q+1}\right)\in\mathbb{P}^3\right|{}^\mathrm t\hspace{-0.5mm}(s,t)\in\mathbb{P}^1\right\},
\end{equation*}
where
$\omega$, $\xi$, and $\eta$ are elements of $\mathbb{F}_{q^2}$ satisfying $\omega^{q+1}=-1$, $\xi^{q+1}=1$ with $\xi^2
\neq-1$, and $\eta^{q+1}=\xi^q+\xi$. Note that $\eta\neq0$ because $\xi^2\neq0$, $-1$.
The curve $C_{F_J}$ is smooth since it is projectively isomorphic to the smooth curve $C_0$.
On the other hand, a complete set of representatives
for ${\rm Aut}(X_I)$ can be taken from ${\rm GL}_4(\mathbb{F}_{q^2})$
(see Lemma \ref{autxi}).
Therefore we have the following.
\begin{cor}\label{cor2}
All nonplanar rational curves of degree $q+1$ on $X_I$ are projectively isomorphic over $\mathbb{F}_{q^2}$ to the smooth curve $C_{F_J}$.
\end{cor}
In the case where $q=2$, we have $|X_I(\mathbb{F}_{q^2})|=45$ where $X_I(\mathbb{F}_{q^2})$ denotes the set of $\mathbb{F}_{q^2}$-rational points of $X_I$,  and ${\rm Aut}(X_I)$ is of order $25920$. Then $|C_F(\mathbb{F}_{q^2})|=5$ for each nonplanar cubic $C_F$ on $X_I$. We can actually obtain by computation $432$ nonplanar cubics on $X_I$ and the stabilizer subgroup of ${\rm Aut}(X_I)$ fixing $C_{F_J}$ of order $60$. By restricting  $X_I$ to $X_I(\mathbb{F}_{q^2})$, we can verify that each cubic intersects $150$ other cubics at a single point, $40$ other cubics at two points and another cubic at five points. Here, when we say two cubics $C_F$, $C_{F'}$ intersect at $n$ points we mean $|C_F(\mathbb{F}_{q^2})\cap C_{F'}(\mathbb{F}_{q^2})|=n$. We can also verify that ${\rm Aut}(X_I)$ acts transitively on $X_I(\mathbb{F}_{q^2})$ and the stabilizer subgroup is of order $576$, and furthermore, there are $48$ cubics passing through each point of $X_I(\mathbb{F}_{q^2})$. These computational data files obtained by using {\tt GAP} \cite{GAP} are available upon request addressed to the author.

We give a brief outline of our paper. In the next section, we prove Theorem \ref{thm0}. By the same argument, we show directly
that each irreducible conic, which is a rational curve of rank $3$, is not contained in $X_A$. In section $3$, we give a bijection between the set $R$ and the quotient of certain sets consisting of invertible $4$-by-$4$ matrices, by showing basic lemmas. In section $4$, we first prove two lemmas which are necessary for our proof of Theorem \ref{thm1}. We prove Theorem \ref{thm1} in the last of the section.

The author is grateful to
Professor Ichiro Shimada for his encouragement during the course of the work and
helpful suggestions on drafts.

\section{Proof of Theorem \ref{thm0}}
\begin{proof}[Proof of Theorem $\ref{thm0}$]
Suppose that a nonplanar rational curve $C_F$ defined by \eqref{lem1} is contained in a smooth $k$-Hermitian surface $X_A$. Denoting by $b_{i,j}$ the entries of the $(d+1)$-by-$(d+1)$ matrix ${}^\mathrm t\hspace{-0.5mm} FAF^{(q)}$, one has the identity
\begin{equation}
\sum_{i,j=0}^{d}b_{i,j}s^{d-i+q(d-j)}t^{i+qj}\equiv0.\label{iden}
\end{equation}
Therefore if $d<q$, all the coefficients $b_{i,j}$ must vanish because the exponents $(i+qj)$'s are all different. This implies that ${}^\mathrm t\hspace{-0.5mm} FAF^{(q)}=O$, but it is a contradiction.
In fact, since ${\rm rank}(F)=4$ by definition, we can take an invertible matrix $F^*$ consisting of linearly independent $4$ column vectors of $F$. Then, however, ${}^\mathrm t\hspace{-0.5mm} F^*A{F^*}^{(q)}$ must be $O$.
If $d=q$, the coefficients $b_{i,j}$ must vanish except for $b_{q,l-1}=-b_{0,l}$ with $1\le l\le q$. This implies that ${\rm rank}({}^\mathrm t\hspace{-0.5mm} FAF^{(q)})\le2$, but it is a contradiction by the argument above. Hence we conclude that
$C_F\not\subset X_A$.

\end{proof}
\begin{rem}
We can similarly give a proof for the case of irreducible conics.
In fact, since an irreducible conic $C_F$ is of rank $3$, we can make an invertible matrix $F^*$ consisting of linearly independent $3$ column vectors of $F$ and a vector linearly independent to those vectors. Suppose that $C_F\subset X_A$. Since $d=2\le q$, one has ${\rm rank}({}^\mathrm t\hspace{-0.5mm} FAF^{(q)})\le2$ in the same argument as the above proof. Therefore the $4$-by-$4$ matrix ${}^\mathrm t\hspace{-0.5mm} {F^*}A{F^*}^{(q)}$ must be of rank $3$ at the most, but ${}^\mathrm t\hspace{-0.5mm} {F^*}A{F^*}^{(q)}$ is of rank $4$ by definition. This is a contradiction. As we have seen, this proof is valid for rational curves which are of rank $\ge3$ and degree $\le q$.
\end{rem}

\section{Basic lemmas}
In this section, we will prove some basic lemmas to prepare for our proof of Theorem \ref{thm1}.
The following lemma gives a necessary and sufficient condition for a nonplanar rational curve of degree $q+1$ to be on a smooth $k$-Hermitian surface.
\begin{lem}\label{lem2}
Let $C_F$ be a nonplanar rational curve of degree $q+1$ defined by \eqref{lem1}.
The curve $C_F$ is contained in a smooth $k$-Hermitian surface $X_A$ if and only if the $(q+2)$-by-$(q+2)$ matrix ${}^\mathrm t\hspace{-0.5mm} FAF^{(q)}$ is of the form
\begin{equation*}
\begin{pmatrix}
0&b_{0,1}&0,\dots,0&0&b_{0,q+1}\\0&b_{1,1}&0,\dots,0&0&b_{1,q+1}\\0&0&0,\dots,0&0&0\\\vdots&\vdots&\vdots&\vdots&\vdots\\0&0&0,\dots,0&0&0\\-b_{0,1}&0&0,\dots,0&-b_{0,q+1}&0\\-b_{1,1}&0&0,\dots,0&-b_{1,q+1}&0
\end{pmatrix}.
\end{equation*}
If the above condition is satisfied, the matrix $F$ is of the form
\begin{equation*}
(\bm{f}_0,\bm{f}_1,\bm{0},\dots,\bm{0},\bm{f}_{q},\bm{f}_{q+1}).
\end{equation*}
\end{lem}
\begin{proof}
As was seen above, the curve $C_F$ is contained in
$X_A$ if and only if one has \eqref{iden}.
In the present case where $d=q+1$, if $C_F\subset X_A$ then
the coefficients $b_{i,j}$ must vanish except for $b_{q,l-1}=-b_{0,l}$, $b_{q+1,l-1}=-b_{1,l}$ with $1\le l\le q+1$.
Since ${\rm rank}(F)=4$, there are $4$ column vectors $\bm{f}_x,\bm{f}_y,\bm{f}_z,\bm{f}_w$ of $F$ with $0\le x<y<z<w\le q+1$ such that the matrix $F^*:=(\bm{f}_x,\bm{f}_y,\bm{f}_z,\bm{f}_w)$ is invertible. Then none of $x,y,z,w$ is from $2$ to $q-1$ because ${}^\mathrm t\hspace{-0.5mm} F^{*}A{F^{*}}^{(q)}$ is also invertible, and thus $x=0,y=1,z=q,w=q+1$. Let $\bm{f}_i$ be the $i$-th column vector with $2\le i\le q-1$ of $F$. Then one has
\begin{equation*}
{}^\mathrm t\hspace{-0.5mm}\bm{f}_iA{F^{*}}^{(q)}=(b_{i,0},b_{i,1},b_{i,q},b_{i,q+1})=(0,0,0,0),
\end{equation*}
and thus $\bm{f}_i=\bm{0}$. Hence $F$ and ${}^\mathrm t\hspace{-0.5mm} FAF^{(q)}$ are of the form described above. The converse is obvious since \eqref{iden} holds automatically.

\end{proof}
A rational curve $C_F$ defined by \eqref{lem1} is also obtained by replacing $F$ by
$\lambda F\varphi(g)$, where $\lambda$ is an element of the multiplicative group $k^{\times}$
and $\varphi$ is a homomorphism from ${\rm GL}_2(k)$ to ${\rm GL}_{d+1}(k)$ defined by the following:
for each ${}^\mathrm t\hspace{-0.5mm}(s,t)\in k^2$ with ${}^\mathrm t\hspace{-0.5mm}(s,t)\neq{}^\mathrm t\hspace{-0.5mm}(0,0)$ and $g\in{\rm GL}_2(k)$, put ${}^\mathrm t\hspace{-0.5mm}(u,v):=g\ {}^\mathrm t\hspace{-0.5mm}(s,t)$, then
\begin{align*}
&\begin{array}{cccc}\varphi:&\hspace{-5mm}{\rm GL}_2(k)&\longrightarrow&{\rm GL}_{d+1}(k)\\
&\hspace{-5mm}\rotatebox{90}{$\in$}&&\rotatebox{90}{$\in$}\\
&\hspace{-5mm}\left(g:{}^\mathrm t\hspace{-0.5mm}(s,t)\mapsto{}^\mathrm t\hspace{-0.5mm}(u,v)\right)&\longmapsto&\left(\varphi(g):{}^\mathrm t\hspace{-0.5mm}(s^d,s^{d-1}t,\dots,t^d)\mapsto
{}^\mathrm t\hspace{-0.5mm}(u^d,u^{d-1}v,\dots,v^d)\right).
\end{array}
\end{align*}
Indeed, it is obvious by definition that $\varphi(I)=I$. Putting ${}^\mathrm t\hspace{-0.5mm}(x,y):=h\ {}^\mathrm t\hspace{-0.5mm}(u,v)$ for each $h\in{\rm GL}_2(k)$, one has
\begin{eqnarray*}
\varphi(hg)\ {}^\mathrm t\hspace{-0.5mm}(s^d,s^{d-1}t,\dots,t^d)&=&{}^\mathrm t\hspace{-0.5mm}({x}^d,{x}^{d-1}y,\dots,{y}^d)\\
&=&\varphi(h)\ {}^\mathrm t\hspace{-0.5mm}(u^d,u^{d-1}v,\dots,v^d)\\
&=&\varphi(h)\varphi(g)\ {}^\mathrm t\hspace{-0.5mm}(s^d,s^{d-1}t,\dots,t^d).
\end{eqnarray*}
Hence $\varphi(hg)=\varphi(h)\varphi(g)$, and thus $\varphi(g)\in{\rm GL}_{d+1}(k)$.

Conversely if there is a matrix $F'$ such that $C_F=C_{F'}$, then one has
\begin{equation*}
F\ {}^\mathrm t\hspace{-0.5mm}(s^d,s^{d-1}t,\dots,st^{d-1},t^d)=F'\ {}^\mathrm t\hspace{-0.5mm}(u^d,u^{d-1}v,\dots,uv^{d-1},v^d)\in\mathbb{P}^3.
\end{equation*}
This implies that there are homogeneous polynomials $f$, $f'$ of degree $d$ such that $f(s,t)=f'(u,v)$.
Therefore there is an element $g$ of ${\rm GL}_2(k)$ such that ${}^\mathrm t\hspace{-0.5mm}(s,t)=g\,{}^\mathrm t\hspace{-0.5mm}(u,v)\in\mathbb{P}^1$, and thus $F'=\lambda F\varphi(g)$ for some $\lambda\in k^{\times}$.
Hence, denoting by ${\rm Im}(\varphi)$ the image of $\varphi$, we see that the set $k^{\times}F{\rm Im}(\varphi)$ corresponds one-to-one with $C_F$.

Let $S$ be the set of matrices $F$ such that
${}^\mathrm t\hspace{-0.5mm} FAF^{(q)}$ satisfies the condition of Lemma \ref{lem2}.
Then by Lemma \ref{lem2}, for each $F\in S$ the set $k^{\times}F{\rm Im}(\varphi)$ corresponds one-to-one with the nonplanar rational curve $C_F$ on $X_A$.
Therefore one has the following bijection
\begin{equation}
k^{\times}\backslash S/{\rm Im}(\varphi)\ni k^{\times}F{\rm Im}(\varphi)\longmapsto C_{F}\in R.\label{bi1}
\end{equation}

By Lemma \ref{lem2}, we define the map
\begin{equation*}
^*:S\ni F=(\bm{f}_0,\bm{f}_1,\bm{0},\dots,\bm{0},\bm{f}_{q},\bm{f}_{q+1})\longmapsto F^*=(\bm{f}_0,\bm{f}_1,\bm{f}_{q},\bm{f}_{q+1})\in S^*,
\end{equation*}
where $S^*$ is written as
\begin{equation*}
S^*=\{F^*\in{\rm GL}_4(k)\ |\ {}^\mathrm t\hspace{-0.5mm} F^*A{F^*}^{(q)}=D_B,\ B\in{\rm GL}_2(k)\},
\end{equation*}
and $D_B$ is a matrix defined by
\begin{equation*}
D_B:=
\begin{pmatrix}
\bm{0}&\bm{b}_1&\bm{0}&\bm{b}_2\\
-\bm{b}_1&\bm{0}&-\bm{b}_2&\bm{0}
\end{pmatrix}
\in{\rm GL}_4(k)\ {\rm for}\ B=(\bm{b}_1,\bm{b}_2)\in{\rm GL}_2(k).
\end{equation*}

Further, we define the map $_*$ from ${\rm Im}(\varphi)\subset{\rm GL}_{q+2}(k)$ to ${{\rm Im}(\varphi)}_*\subset{\rm GL}_4(k)$
as follows:
\begin{align*}
&{\rm for}\ {\rm every}\ g=\begin{pmatrix}\alpha&\beta\\\gamma&\delta\end{pmatrix}\in{\rm GL}_2(k),\\
&\varphi(g)\hspace{-1mm}=\hspace{-1mm}
\begin{pmatrix}
\alpha^{q+1}&\alpha^q\beta&,\dots,&\alpha\beta^q&\beta^{q+1}\\
\alpha^q\gamma&\alpha^q\delta&,\dots,&\gamma\beta^q&\delta\beta^q\\
\vdots&\vdots&\vdots&\vdots&\vdots\\
\alpha\gamma^q&\beta\gamma^q&,\dots,&\alpha\delta^q&\beta\delta^q\\
\gamma^{q+1}&\delta\gamma^q&,\dots,&\gamma\delta^q&\delta^{q+1}
\end{pmatrix}
\mapsto
\varphi(g)_*\hspace{-1mm}=\hspace{-1mm}
\begin{pmatrix}
\alpha^{q+1}&\alpha^q\beta&\alpha\beta^q&\beta^{q+1}\\
\alpha^q\gamma&\alpha^q\delta&\gamma\beta^q&\delta\beta^q\\
\alpha\gamma^q&\beta\gamma^q&\alpha\delta^q&\beta\delta^q\\
\gamma^{q+1}&\delta\gamma^q&\gamma\delta^q&\delta^{q+1}
\end{pmatrix},
\end{align*}
where ${{\rm Im}(\varphi)}_*$ is written as
\begin{equation*}
{{\rm Im}(\varphi)}_*=
\left\{\left.
\begin{pmatrix}
\alpha^{q}g&\beta^{q}g\\
\gamma^{q}g&\delta^{q}g\\
\end{pmatrix}
\in{\rm GL}_4(k)\ \right|\ 
g\in{\rm GL}_2(k)
\right\}.
\end{equation*}
Indeed, it is easy to see that
${\rm det}(\varphi(g)_*)={\rm det}(g)^{2q+2}$
for every $g\in{\rm GL}_2(k)$, and thus $\varphi(g)_*\in{\rm GL}_4(k)$.

We denote by $\varphi_*$ the composition of $\varphi$ and $_*$, namely $\varphi_*(g)=\varphi(g)_*$ for every $g\in{\rm GL}_2(k)$.
\begin{lem}\label{lem3}
The map $\varphi_*$
is a homomorphism from ${\rm GL}_2(k)$ to ${\rm GL}_4(k)$.
There is the following natural bijection
\begin{equation*}
k^{\times}\backslash S/{\rm Im}(\varphi)\longrightarrow
k^{\times}\backslash S^*/{{\rm Im}(\varphi)}_*.
\end{equation*}
\end{lem}
\begin{proof}
For each 
\begin{equation*}
g=\begin{pmatrix}\alpha&\beta\\\gamma&\delta\end{pmatrix},\ h=\begin{pmatrix}x&y\\z&w\end{pmatrix}\in{\rm GL}_2(k),
\end{equation*}
one has
\begin{equation*}
gh=\begin{pmatrix}\alpha x+\beta z&\alpha y+\beta w\\\gamma x+\delta z&\gamma y+\delta w\end{pmatrix}.
\end{equation*}
Therefore
\begin{equation*}
\varphi_*(gh)=
\begin{pmatrix}(\alpha x+\beta z)^qgh&(\alpha y+\beta w)^qgh\\(\gamma x+\delta z)^qgh&(\gamma y+\delta w)^qgh
\end{pmatrix}.
\end{equation*}
On the other hand,
\begin{eqnarray*}
\varphi_*(g)\varphi_*(h)&=&
\begin{pmatrix}\alpha^q g&\beta^q g\\\gamma^q g&\delta^q g\end{pmatrix}\begin{pmatrix}x^qh&y^qh\\z^qh&w^qh\end{pmatrix}\\
&=&\begin{pmatrix}\alpha^qx^qgh+\beta^qz^qgh&\alpha^qy^qgh+\beta^qw^qgh\\\gamma^qx^qgh+\delta^qz^qgh&\gamma^qy^qgh+\delta^qw^qgh \end{pmatrix}\\
&=&\begin{pmatrix}(\alpha^q{x}^q+\beta^q{z}^q)gh&(\alpha^q{y}^q+\beta^q{w}^q)gh\\(\gamma^q{x}^q+\delta^q{z}^q)gh&(\gamma^q{y}^q+\delta^q{w}^q)gh
\end{pmatrix}.
\end{eqnarray*}
Since the $q$-th power is an automorphism of $k$, one has $\varphi_*(gh)=\varphi_*(g)\varphi_*(h)$ and thus $\varphi_*$ is a homomorphism from ${\rm GL}_2(k)$ to ${\rm GL}_4(k)$.

For each $F\in S$,\ $g\in{\rm GL}_2(k)$, denoting by $a_{i,j}$ the entries of $\varphi(g)$, we can write the $j$-th column vector $\bm{g}_j$ with $j\in\{0,1,q,q+1\}$ of $F\varphi(g)$ as
\begin{equation*}
\bm{g}_j=\sum_{i\in\{0,1,q,q+1\}}a_{i,j}\bm{f}_i,
\end{equation*}
since $\bm{f}_i=\bm{0}$ for $2\le i\le q-1$. Then it is immediate from definition that
\begin{equation*}
F^*\varphi_*(g)=(\bm{g}_0,\bm{g}_1,\bm{g}_{q},\bm{g}_{q+1}),
\end{equation*}
and thus ${(F\varphi(g))}^*=F^*\varphi_*(g)$. This implies that there is the natural map from $k^{\times}\backslash S/{\rm Im}(\varphi)$ to $k^{\times}\backslash S^*/{{\rm Im}(\varphi)}_*$.
The bijectivity is obvious
since by definition the map $S\rightarrow S^*$ is bijective.

\end{proof}
By \eqref{bi1} and Lemma \ref{lem3}, one has the
bijection
\begin{equation}
k^{\times}\backslash S^*/{{\rm Im}(\varphi)}_*\ni k^{\times}F^*{{\rm Im}(\varphi)}_*\longmapsto C_F\in R.\label{bi2}
\end{equation}

The following well-known proposition is useful. The readers may find a proof for example in \cite{Be} and \cite[Proposition 2.5.]{Sh1}.
\begin{prop}\label{prop1}
For each element $A$ of ${\rm GL}_n(k)$, there is an element $B$ of ${\rm GL}_n(k)$ such that $A={}^\mathrm t\hspace{-0.5mm} BB^{(q)}$. If $A$ is a Hermitian matrix, then the matrix $B$ can be taken from ${\rm GL}_n(\mathbb{F}_{q^2})$.
\end{prop}
By Proposition \ref{prop1}, it follows immediately that a smooth $k$-Hermitian (resp. Hermitian) surface is projectively isomorphic over $k$ (resp. $\mathbb{F}_{q^2}$) to the Fermat surface $X_I$.

We define the set
\begin{equation*}
M:=\left\{\left.
D_B:=
\begin{pmatrix}
\bm{0}&\bm{b}_1&\bm{0}&\bm{b}_2\\
-\bm{b}_1&\bm{0}&-\bm{b}_2&\bm{0}
\end{pmatrix}
\in{\rm GL}_4(k)\right|
B=
\begin{pmatrix}
\bm{b}_1&\bm{b}_2
\end{pmatrix}
\in{\rm GL}_2(k)
\right\}.
\end{equation*}
Then the following map is surjective:
\begin{equation}
S^*\ni F^*\longmapsto {}^\mathrm t\hspace{-0.5mm} F^*A{F^*}^{(q)}\in M.\label{surjM}
\end{equation}
In fact, by Proposition \ref{prop1} there is an element $D$ of ${\rm GL}_4(k)$ such that $D_B={}^\mathrm t\hspace{-0.5mm} DD^{(q)}$ for each $D_B\in M$. Similarly there is an element $A'$ of ${\rm GL}_4(k)$ such that $A={}^\mathrm t\hspace{-0.5mm} A'A'^{(q)}$. Hence putting $F^*:=A'^{-1}D$, one has ${}^\mathrm t\hspace{-0.5mm} F^*A{F^*}^{(q)}=D_B$, and thus $F^*\in S^*$.

\begin{lem}\label{rc0}
The set $R$ is nonempty, and each element of $R$ is projectively isomorphic over $k$ to the smooth curve 
\begin{equation*}
C_0:=\left\{{}^\mathrm t\hspace{-0.5mm}(s^{q+1},s^qt,st^q,t^{q+1})\in\mathbb{P}^3\ |\ {}^\mathrm t\hspace{-0.5mm}(s,t)\in\mathbb{P}^1\right\}.
\end{equation*} 
\end{lem}
\begin{proof}
The set $S^*$ is nonempty by the surjectivity of the map \eqref{surjM}.
Hence by \eqref{bi2} the set $R$ is nonempty.
For each element $C_F$ of $R$, it is obvious by definition that
\begin{equation*}
{F^*}^{-1}F=(\bm{e}_1,\bm{e}_2,\bm{0},\dots,\bm{0},\bm{e}_3,\bm{e}_4)\ {\rm with}\ (\bm{e}_1,\bm{e}_2,\bm{e}_3,\bm{e}_4)=I.
\end{equation*}
This implies that
$C_F$ is projectively isomorphic over $k$ to $C_0$.
Then by definition, the curve $C_0$ is smooth clearly.

\end{proof}
\begin{rem}
It is known that each nonplanar nonreflexive curve of degree $q+1$
is projectively isomorphic to the curve $C_0$ $(${\rm cf}.\ \cite[Theorem 2]{BH}$)$. For nonreflexive curves, see also \cite{He}.
Hence by Lemma $\ref{rc0}$, each element of $R$ is projectively isomorphic to each nonplanar nonreflexive curve of degree $q+1$.
\end{rem}
\begin{rem}
In the case where $A=I$, we can find an element of $R$.
We put
\begin{equation*}
J:=\begin{pmatrix}0&-1\\1&0\end{pmatrix}.
\end{equation*}
Then the matrix $D_J$
is a Hermitian matrix. Hence by Proposition $\ref{prop1}$, there is an element ${F_J}^*$ of ${\rm GL}_4(\mathbb{F}_{q^2})$ such that ${}^\mathrm t\hspace{-0.5mm}{F_J}^*{{F_J}^*}^{(q)}=D_J$.
Actually taking ${F_J}^*$ such as
\begin{equation*}
\begin{pmatrix}
\eta^{-q}\xi^q&0&0&-\eta^{-q}\\
0&1&0&0\\
0&0&1&0\\
\omega\eta^{-1}\xi&0&0&\omega\eta^{-1}
\end{pmatrix}
\end{equation*}
for $\omega$, $\xi$ and $\eta$ as mentioned in Introduction, one has by \eqref{bi2} the corresponding curve $C_{F_J}$ lying on $X_I$.
\end{rem}

\section{Proof of Theorem \ref{thm1}}
The group ${\rm Aut}(X_A)$ of projective automorphisms of $X_A$ is equal to
\begin{equation*}
\{Q\in{\rm GL}_4(k)\ |\ {}^\mathrm t\hspace{-0.5mm} QAQ^{(q)}=\lambda A,\ \lambda\in k^{\times}\}/k^{\times}I.
\end{equation*}
By Proposition \ref{prop1}, the group ${\rm Aut}(X_A)$ is conjugate to ${\rm Aut}(X_I)$ in ${\rm PGL}_4(k)$.

We prove the following lemma on matrix groups of arbitrary rank because we need the lemma to our proof of Theorem \ref{thm1}. 
\begin{lem}\label{autxi}
Let $n$ be a positive integer. The group
${\rm PGU}_n(\mathbb{F}_{q^2})$ is isomorphic to
\begin{equation*}
G:=\{Q\in{\rm GL}_n(k)\ |\ {}^\mathrm t\hspace{-0.5mm} QQ^{(q)}=\lambda I,\ \lambda\in k^{\times}\}/k^{\times}I.
\end{equation*}
\end{lem}
\begin{proof}
We consider the map
\begin{equation*}
G\ni Qk^{\times}\longmapsto\xi_{\lambda}Q\bm{\mu}_{q+1}\in{\rm PGU}_n(\mathbb{F}_{q^2}),
\end{equation*}
where $\lambda$ is the element of $k^{\times}$ satisfying ${}^\mathrm t\hspace{-0.5mm} QQ^{(q)}=\lambda I$ and $\xi_{\lambda}$ is an element of $k^{\times}$ satisfying 
${\xi_{\lambda}}^{q+1}=\lambda^{-1}$. Then the map is well-defined. In fact, it is obvious that ${}^\mathrm t\hspace{-0.5mm}(\xi_{\lambda}Q)(\xi_{\lambda}Q)^{(q)}=I$,
and the matrix $\xi_{\lambda}Q$ has the entries in $\mathbb{F}_{q^2}$ because $I$ is a Hermitian matrix.
Hence $\xi_{\lambda}Q\bm{\mu}_{q+1}$ belongs to ${\rm PGU}_n(\mathbb{F}_{q^2})$. Further, putting $P:=\alpha Q$ for each $\alpha\in k^{\times}$, one has ${}^\mathrm t\hspace{-0.5mm} PP^{(q)}=\alpha^{q+1}\lambda I$. It is easily shown by definition that
\begin{equation*}
\xi_{\alpha^{q+1}\lambda}\bm{\mu}_{q+1}=\xi_{\alpha^{q+1}}\xi_{\lambda}\bm{\mu}_{q+1}\ \ {\rm and}\ \ \alpha\xi_{\alpha^{q+1}}\bm{\mu}_{q+1}=\bm{\mu}_{q+1}.
\end{equation*}
Therefore we conclude that
\begin{equation*}
\xi_{\alpha^{q+1}\lambda} P\bm{\mu}_{q+1}=\xi_{\lambda}Q\bm{\mu}_{q+1}.
\end{equation*}
Thus the map is independent of the choice of representatives for $G$.

Let $Q'k^{\times}$ be an element of $G$ with ${}^\mathrm t\hspace{-0.5mm}{Q'}{Q'}^{(q)}=\eta I$ for some $\eta\in k^{\times}$. Then one has 
\begin{equation*}
(\xi_{\eta}Q'\bm{\mu}_{q+1})(\xi_{\lambda}Q\bm{\mu}_{q+1})=\xi_{\eta\lambda}Q'Q\bm{\mu}_{q+1},
\end{equation*}
since $\xi_{\eta}\xi_{\lambda}\bm{\mu}_{q+1}=\xi_{\eta\lambda}\bm{\mu}_{q+1}$.
Hence the map is a homomorphism from $G$ to ${\rm PGU}_n(\mathbb{F}_{q^2})$. The injectivity and the surjectivity are immediate from definition.

\end{proof}
By Lemma \ref{autxi}, the group ${\rm Aut}(X_A)$ isomorphic to ${\rm PGU}_4(\mathbb{F}_{q^2})$.

The following lemma is a key ingredient in our proof of Theorem \ref{thm1}.
\begin{lem}\label{lemM}
For every $g$, $B\in{\rm GL}_2(k)$, one has
\begin{equation*}
{}^\mathrm t\hspace{-0.5mm}\varphi_*(g)D_B\varphi_*(g)^{(q)}={\rm det}(g)^qD_{{}^\mathrm t\hspace{-0.5mm} gBg^{(q^2)}}.
\end{equation*}
\end{lem}
\begin{proof}
The proof is due to straightforward computation. 
We put
\begin{equation*}
g:=
\begin{pmatrix}
\alpha&\beta\\
\gamma&\delta
\end{pmatrix},
\ 
B:=(\bm{b}_1,\bm{b}_2).
\end{equation*}
Then one has
\begin{eqnarray*}
\lefteqn{{}^\mathrm t\hspace{-0.5mm}\varphi_*(g)D_B\varphi_*(g)^{(q)}}\\
&=&
\begin{pmatrix}
\alpha^{q}\ {}^\mathrm t\hspace{-0.5mm} g&\gamma^{q}\ {}^\mathrm t\hspace{-0.5mm} g\\
\beta^{q}\ {}^\mathrm t\hspace{-0.5mm} g&\delta^{q}\ {}^\mathrm t\hspace{-0.5mm} g\\
\end{pmatrix}
\begin{pmatrix}
\bm{0}&\bm{b}_1&\bm{0}&\bm{b}_2\\
-\bm{b}_1&\bm{0}&-\bm{b}_2&\bm{0}
\end{pmatrix}
\begin{pmatrix}
\alpha^{q^2}g^{(q)}&\beta^{q^2}g^{(q)}\\
\gamma^{q^2}g^{(q)}&\delta^{q^2}g^{(q)}\\
\end{pmatrix}\\
&=&
\begin{pmatrix}
-\gamma^q\ {}^\mathrm t\hspace{-0.5mm} g\bm{b}_1&\alpha^q\ {}^\mathrm t\hspace{-0.5mm} g\bm{b}_1&-\gamma^q\ {}^\mathrm t\hspace{-0.5mm} g\bm{b}_2&\alpha^q\ {}^\mathrm t\hspace{-0.5mm} g\bm{b}_2\\
-\delta^q\ {}^\mathrm t\hspace{-0.5mm} g\bm{b}_1&\beta^q\ {}^\mathrm t\hspace{-0.5mm} g\bm{b}_1&-\delta^q\ {}^\mathrm t\hspace{-0.5mm} g\bm{b}_2&\beta^q\ {}^\mathrm t\hspace{-0.5mm} g\bm{b}_2\\
\end{pmatrix}
\begin{pmatrix}
\alpha^{q^2+q}&\alpha^{q^2}\beta^q&\alpha^q\beta^{q^2}&\beta^{q^2+q}\\
\alpha^{q^2}\gamma^q&\alpha^{q^2}\delta^q&\gamma^q\beta^{q^2}&\delta^q\beta^{q^2}\\
\alpha^q\gamma^{q^2}&\beta^q\gamma^{q^2}&\alpha^q\delta^{q^2}&\beta^q\delta^{q^2}\\
\gamma^{q^2+q}&\delta^q\gamma^{q^2}&\gamma^q\delta^{q^2}&\delta^{q^2+q}
\end{pmatrix}.
\end{eqnarray*}
Putting
\begin{equation*}
{}^\mathrm t\hspace{-0.5mm}\varphi_*(g)D_B\varphi_*(g)^{(q)}
:=
\begin{pmatrix}
\bm{c}_1&\bm{c}_2&\bm{c}_3&\bm{c}_4\\
\bm{c}_5&\bm{c}_6&\bm{c}_7&\bm{c}_8
\end{pmatrix},
\end{equation*}
one has
\begin{eqnarray*}
\bm{c}_1&=&-\alpha^{q^2+q}\gamma^q\ {}^\mathrm t\hspace{-0.5mm} g\bm{b}_1+\alpha^{q^2}\gamma^q\alpha^q\ {}^\mathrm t\hspace{-0.5mm} g\bm{b}_1-\alpha^q\gamma^{q^2}\gamma^q\ {}^\mathrm t\hspace{-0.5mm} g\bm{b}_2+\gamma^{q^2+q}\alpha^q\ {}^\mathrm t\hspace{-0.5mm} g\bm{b}_2\\
&=&\bm{0},\\
\bm{c}_2&=&-\alpha^{q^2}\beta^q\gamma^q\ {}^\mathrm t\hspace{-0.5mm} g\bm{b}_1+\alpha^{q^2}\delta^q\alpha^q\ {}^\mathrm t\hspace{-0.5mm} g\bm{b}_1-\beta^q\gamma^{q^2}\gamma^q\ {}^\mathrm t\hspace{-0.5mm} g\bm{b}_2+\delta^q\gamma^{q^2}\alpha^q\ {}^\mathrm t\hspace{-0.5mm} g\bm{b}_2\\
&=&{\rm det}(g)^q(\alpha^{q^2}\ {}^\mathrm t\hspace{-0.5mm} g\bm{b}_1+\gamma^{q^2}\ {}^\mathrm t\hspace{-0.5mm} g\bm{b}_2)\\
&=&{\rm det}(g)^q\ {}^\mathrm t\hspace{-0.5mm} g(\bm{b}_1,\bm{b}_2)\ {}^\mathrm t\hspace{-0.5mm}(\alpha^{q^2},\gamma^{q^2}),\\
\bm{c}_3&=&-\alpha^q\beta^{q^2}\gamma^q\ {}^\mathrm t\hspace{-0.5mm} g\bm{b}_1+\gamma^q\beta^{q^2}\alpha^q\ {}^\mathrm t\hspace{-0.5mm} g\bm{b}_1-\alpha^q\delta^{q^2}\gamma^q\ {}^\mathrm t\hspace{-0.5mm} g\bm{b}_2+\gamma^q\delta^{q^2}\alpha^q\ {}^\mathrm t\hspace{-0.5mm} g\bm{b}_2\\
&=&\bm{0},\\
\bm{c}_4&=&-\beta^{q^2+q}\gamma^q\ {}^\mathrm t\hspace{-0.5mm} g\bm{b}_1+\delta^q\beta^{q^2}\alpha^q\ {}^\mathrm t\hspace{-0.5mm} g\bm{b}_1-\beta^q\delta^{q^2}\gamma^q\ {}^\mathrm t\hspace{-0.5mm} g\bm{b}_2+\delta^{q^2+q}\alpha^q\ {}^\mathrm t\hspace{-0.5mm} g\bm{b}_2\\
&=&{\rm det}(g)^q(\beta^{q^2}\ {}^\mathrm t\hspace{-0.5mm} g\bm{b}_1+\delta^{q^2}\ {}^\mathrm t\hspace{-0.5mm} g\bm{b}_2)\\
&=&{\rm det}(g)^q\ {}^\mathrm t\hspace{-0.5mm} g(\bm{b}_1,\bm{b}_2)\ {}^\mathrm t\hspace{-0.5mm}(\beta^{q^2},\delta^{q^2}),\\
\bm{c}_5&=&-\alpha^{q^2+q}\delta^q\ {}^\mathrm t\hspace{-0.5mm} g\bm{b}_1+\alpha^{q^2}\gamma^q\beta^q\ {}^\mathrm t\hspace{-0.5mm} g\bm{b}_1-\alpha^q\gamma^{q^2}\delta^q\ {}^\mathrm t\hspace{-0.5mm} g\bm{b}_2+\gamma^{q^2+q}\beta^q\ {}^\mathrm t\hspace{-0.5mm} g\bm{b}_2\\
&=&-{\rm det}(g)^q(\alpha^{q^2}\ {}^\mathrm t\hspace{-0.5mm} g\bm{b}_1+\gamma^{q^2}\ {}^\mathrm t\hspace{-0.5mm} g\bm{b}_2)\\
&=&-{\rm det}(g)^q\ {}^\mathrm t\hspace{-0.5mm} g(\bm{b}_1,\bm{b}_2)\ {}^\mathrm t\hspace{-0.5mm}(\alpha^{q^2},\gamma^{q^2}),\\
\bm{c}_6&=&-\alpha^{q^2}\beta^q\delta^q\ {}^\mathrm t\hspace{-0.5mm} g\bm{b}_1+\alpha^{q^2}\delta^q\beta^q\ {}^\mathrm t\hspace{-0.5mm} g\bm{b}_1-\beta^q\gamma^{q^2}\delta^q\ {}^\mathrm t\hspace{-0.5mm} g\bm{b}_2+\delta^q\gamma^{q^2}\beta^q\ {}^\mathrm t\hspace{-0.5mm} g\bm{b}_2\\
&=&\bm{0},\\
\bm{c}_7&=&-\alpha^q\beta^{q^2}\delta^q\ {}^\mathrm t\hspace{-0.5mm} g\bm{b}_1+\gamma^q\beta^{q^2}\beta^q\ {}^\mathrm t\hspace{-0.5mm} g\bm{b}_1-\alpha^q\delta^{q^2}\delta^q\ {}^\mathrm t\hspace{-0.5mm} g\bm{b}_2+\gamma^q\delta^{q^2}\beta^q\ {}^\mathrm t\hspace{-0.5mm} g\bm{b}_2\\
&=&-{\rm det}(g)^q(\beta^{q^2}\ {}^\mathrm t\hspace{-0.5mm} g\bm{b}_1+\delta^{q^2}\ {}^\mathrm t\hspace{-0.5mm} g\bm{b}_2)\\
&=&-{\rm det}(g)^q\ {}^\mathrm t\hspace{-0.5mm} g(\bm{b}_1,\bm{b}_2)\ {}^\mathrm t\hspace{-0.5mm}(\beta^{q^2},\delta^{q^2}),\\
\bm{c}_8&=&-\beta^{q^2+q}\delta^q\ {}^\mathrm t\hspace{-0.5mm} g\bm{b}_1+\delta^q\beta^{q^2}\beta^q\ {}^\mathrm t\hspace{-0.5mm} g\bm{b}_1-\beta^q\delta^{q^2}\delta^q\ {}^\mathrm t\hspace{-0.5mm} g\bm{b}_2+\delta^{q^2+q}\beta^q\ {}^\mathrm t\hspace{-0.5mm} g\bm{b}_2\\
&=&\bm{0}.
\end{eqnarray*}
Hence one has
\begin{eqnarray*}
(\bm{c}_2,\bm{c}_4)={\rm det}(g)^q\ {}^\mathrm t\hspace{-0.5mm} gB
g^{(q^2)}=-(\bm{c}_5,\bm{c}_7),\ \ 
\bm{c}_1=\bm{c}_3=\bm{c}_6=\bm{c}_8=\bm{0}.
\end{eqnarray*}
This completes the proof.

\end{proof}

\begin{proof}[Proof of Theorem $\ref{thm1}$]
We define an equivalence relation $\sim$ on the set $M$
as follows: $D_{B}\sim D_{B'}$ for $D_{B}$, $D_{B'}\in M$ if there is an element $g\in {\rm GL}_2(k)$ such that $D_{B'}={}^\mathrm t\hspace{-0.5mm}\varphi_*(g)D_{B}{\varphi_*(g)}^{(q)}$. We denote by ${D_{B}}^{\varphi_*}$ an equivalence class containing $D_{B}$.
On the other hand, the group ${\rm Aut}(X_A)$ acts on $k^{\times}\backslash S^*/{{\rm Im}(\varphi)}_*$ by multiplication from the left. Then the following map is bijective:
\begin{align*}
\begin{array}{ccc}
{\rm Aut}(X_A)k^{\times}\backslash S^*/{{\rm Im}(\varphi)}_*&\longrightarrow& k^{\times}\backslash M/\sim\\
\rotatebox{90}{$\in$}&&\rotatebox{90}{$\in$}\\
{\rm Aut}(X_A)k^{\times}F^*{{\rm Im}(\varphi)}_*&\longmapsto&k^{\times}({}^\mathrm t\hspace{-0.5mm}{F^*}A{F^*}^{(q)})^{\varphi_*}.
\end{array}
\end{align*}
Indeed, the surjectivity is obvious
since the map \eqref{surjM} is surjective.
If we assume that $k^{\times}({}^\mathrm t\hspace{-0.5mm} F^*A{F^*}^{(q)})^{\varphi_*}=k^{\times}({}^\mathrm t\hspace{-0.5mm}{F_1}^*A{{F_1}^*}^{(q)})^{\varphi_*}$ for some ${F_1}^*\in S^*$, then we have
\begin{equation*}
{}^\mathrm t\hspace{-0.5mm}{({F_1}^*\varphi_*(g){{F}^*}^{-1})}A{({F_1}^*\varphi_*(g){{F}^*}^{-1})}^{(q)}=\lambda A
\end{equation*}
for some $g\in{\rm GL}_2(k)$ and $\lambda\in k^{\times}$. Therefore $k^{\times}{F_1}^*\varphi_*(g){{F}^*}^{-1}$ belongs to ${\rm Aut}(X_A)$. This implies the injectivity, and thus bijectivity. By Proposition \ref{prop1}, there is an element $B'$ of ${\rm GL}_2(k)$ such that $B={}^\mathrm t\hspace{-0.5mm} B'{B'}^{(q^2)}$ for each $D_B\in M$. Then by Lemma \ref{lemM}, one has
\begin{equation*}
{}^\mathrm t\hspace{-0.5mm}\varphi_*({B'}^{-1})D_B{\varphi_*({B'}^{-1})}^{(q)}={{\rm det}({B'}^{-1})}^qD_I.
\end{equation*}
This implies that $k^{\times}{D_{B}}^{\varphi_*}=k^{\times}{D_{I}}^{\varphi_*}$.
Hence $|k^{\times}\backslash M/\sim|=1$ and thus
$|{\rm Aut}(X_A)k^{\times}\backslash S^*/{{\rm Im}(\varphi)}_*|=1$,
and by \eqref{bi2} one has $|{\rm Aut}(X_A)\backslash R|=1$. This proves half of our theorem.

Let $\varGamma/k^{\times}I$ be the stabilizer subgroup
of ${\rm Aut}(X_A)$
fixing the element $k^{\times}{F_I}^*{{\rm Im}(\varphi)}_*$ of $k^{\times}\backslash S^*/{{\rm Im}(\varphi)}_*$
such that ${}^\mathrm t\hspace{-0.5mm}{F_I}^*A{{F_I}^*}^{(q)}=D_I$. Then it follows immediately that
\begin{equation*}
\varGamma
={F_I}^*{{\rm Im}(\varphi)}_*{{F_I}^*}^{-1}
\cap
\{Q\in{\rm GL}_4(k)\ |\ {}^\mathrm t\hspace{-0.5mm} QAQ^{(q)}=\lambda A,\ \lambda\in k^{\times}\}.
\end{equation*}
Hence each element of $\varGamma$ can be written as ${F_I}^*\varphi_*(g){{F_I}^*}^{-1}$ for some element $g$ of ${\rm GL}_2(k)$ satisfying
\begin{equation*}
{}^\mathrm t\hspace{-0.5mm}({F_I}^*\varphi_*(g){{F_I}^*}^{-1})A({F_I}^*\varphi_*(g){{F_I}^*}^{-1})^{(q)}=\lambda A\ \ {\rm for}\ \lambda\in k^{\times},
\end{equation*}
or equivalently,
\begin{equation*}
{}^\mathrm t\hspace{-0.5mm}\varphi_*(g)D_I\varphi_*(g)^{(q)}=\lambda D_I\ \ {\rm for}\ \lambda\in k^{\times}.
\end{equation*}
By Lemma \ref{lemM}, this equality is equivalent to
${}^\mathrm t\hspace{-0.5mm} gg^{(q^2)}=\lambda I$ for $\lambda\in k^{\times}$. Consequently, one has the following isomorphism:
\begin{align*}
\begin{array}{ccc}
\{g\in{\rm GL}_2(k)\ |\ {}^\mathrm t\hspace{-0.5mm} gg^{(q^2)}=\lambda I,\ \lambda\in k^{\times}\}/k^{\times}I&\longrightarrow&\varGamma/k^{\times}I\\
\rotatebox{90}{$\in$}&&\rotatebox{90}{$\in$}\\
 gk^{\times}&\longmapsto&{F_I}^*\varphi_*(g){{F_I}^*}^{-1}k^{\times}.
\end{array}
\end{align*}
By Lemma \ref{autxi},
we conclude that ${\rm PGU}_2(\mathbb{F}_{q^4})\simeq\varGamma/k^{\times}I$.

\end{proof}


\end{document}